\documentclass{amsart}

\usepackage{amsmath}
\usepackage{amssymb}
\usepackage{amsthm}
\usepackage{mathtools}
\usepackage[shortlabels]{enumitem}

\usepackage{tikz}
\usepackage{subcaption}

\usepackage{hyperref}

\theoremstyle{plain}
\newtheorem{theorem}{Theorem}
\newtheorem{lemma}[theorem]{Lemma}
\newtheorem{cor}[theorem]{Corollary}
\newtheorem*{theorem*}{Theorem}

\usepackage{shortcuts}

\begin{document}

%\title{An X-ray Inversion Formula on Constant Curvature Anosov Surfaces}
\title{An inversion formula for the X-ray normal operator over closed hyperbolic surfaces}
\author{Sean Richardson}
%\address{ADDRESS GOES HERE}
\email{seanhr@uw.edu}
\pagestyle{plain}

%WITH SYMBOLS:
%\begin{abstract}
%	We construct an explicit inversion formulas for Guillarmou's normal operator $\pi_*\Pi\pi^*$ on an Anosov surface of constant curvature. This $\Pi$ operator is defined as a weak limit $\Pi^{(z)} \to \Pi$ as $z \to 0$ for an ``attenuated X-ray transform'' operator $\Pi^{(z)}$, and we prove our result by constructing an inversion formula for $\pi_*\Pi^{(z)}\pi^*$ on both the Poincar\'e disk and surfaces of constant negative curvature.
%\end{abstract}
\begin{abstract}
	We construct an explicit inversion formula for Guillarmou's normal operator on closed surfaces of constant negative curvature. This normal operator can be defined as a weak limit for an ``attenuated normal operator'', and we prove this inversion formula by first constructing an additional inversion formula for this attenuated normal operator on both the Poincar\'e disk and closed surfaces of constant negative curvature. A consequence of the inversion formula is the explicit construction of invariant distributions with prescribed pushforward over closed hyperbolic manifolds.
\end{abstract}

\maketitle

\section{Introduction}
%focus on only the preliminaries necessary to understand theorem statements

%note it is known that these are injective
We consider connected closed surfaces $(M,g)$ of constant negative curvature with set $\mc{G}$ of  closed geodesics with unit-speed. Such surfaces have \emph{X-ray transform} $I_0: C^{\infty}(M) \to \ell^{\infty}(\mc{G})$ defined by integrating over each closed geodesic by
\begin{equation*}
(I_0 f)(\gamma) \defeq \int_{0}^{\operatorname{L}(\gamma)} f(\gamma(t))dt, 
\quad f \in C^{\infty}(M), 
\quad \gamma \in \mc{G}
\end{equation*}
where $\operatorname{L}(\gamma)$ denotes the length of $\gamma$. For such surfaces, there is a unique unit-speed closed geodesic $\gamma \in \mc{G}$ in each free homotopy class. In fact, closed geodesics are parametrized by the free homotopy classes for all \emph{Anosov manifolds}, defined in Section~\ref{sec:anosov}, which in particular includes all closed manifolds of strictly negative sectional curvature. This same definition of the X-ray transform extends to all Anosov manifolds and a natural point of study is the injectivity of this operator. 

%TODO? check this history...
Guillemin and Kazhdan \cite{guillemin1980some} showed injectivity for surfaces with non-positive curvature, Croke and Sharafutdinov \cite{croke1998spectral} generalized this to higher dimensional manifolds with non-positive sectional curvature, then Dairbekov and Sharafutdinov \cite{dairbekov2003some} obtained injectivity of $I_0$ for Anosov manifolds in general. Guillarmou \cite{guillarmou2017invariant} provided a microlocal proof of the Liv\v sic theorem \cite{livsic1972cohomology} which gives a proof of the injectivity of $I_0$ on an Anosov manifold $M$ that generalizes to the case of tensors and has applications to the classical question of spectral rigidity \cite{guillarmou2019marked}. In particular, Guillarmou defined an important operator $\Pi: C^{\infty}(SM) \to C^{\infty}(SM)$ on the unit sphere bundle $\pi_0: SM \to M$ (see Section~\ref{sec:pre}) and shows injectivity by proving $I_0 f = 0$ implies $\Pi \pi_0^* f = 0$, which then implies $f = 0$. That is, the operator $\Pi \pi_0^*$ is injective (over functions with average value $0$) and the subject of this paper is an explicit inversion formula for $\Pi \pi_0^*$ on closed surfaces of constant negative Gaussian curvature $K < 0$. In fact, the result is slightly stronger, for we invert Guillarmou's \emph{normal operator} $\Pi_0 = \pi_{0*}\Pi\pi_0^*$ by defining the operator
\begin{equation}
	S_{K}f \defeq \int_{S_xM}\int_{0}^{\infty}e^{-\sqrt{-K} \cdot t}f(\gamma_{x,v}(t))dtdS_x(v),
	\quad f \in C^{\infty}(M)
	\label{eq:S}
\end{equation}
where $\gamma_{x,v}(t)$ denotes the unique geodesic on $M$ satisfying $\gamma_{x,v}(0) = x$, $\dot{\gamma}_{x,v}(0) = v$, and the measure $dS_x$ on fiber $S_xM$ is defined in Section~\ref{sec:SM}. Applying this operator followed by the Laplace-Beltrami operator $\Delta$ gives the following inversion formula.
\begin{theorem}
	Given a smooth function $f$ with average value zero over a connected closed surface $M$ of constant curvature $K < 0$ we have the inversion formula
	\begin{equation}
		\Delta S_K \Pi_0 f = -8\pi^2 f.
	\end{equation}
\label{thm:inv-formula}
\end{theorem}
Observe the above theorem reduces the problem of reconstructing $f$ from $I_0 f$ to the problem of constructing $\Pi_0 f$ from $I_0 f$ for surfaces of constant negative curvature. Furthermore, this inversion formula provides an explicit construction of invariant distributions with prescribed pushforward in the case of a closed hyperbolic surface $M$. That is, given any $f \in C^{\infty}(M)$, we construct a distribution $w \in H^s(M)$ for all $s < 0$ such that $Xw = 0$ and $\pi_{0*}w = f$,
%For example, if $f \in C^{\infty}(M)$ can be written as $\Pi_0 g = f$ for $g \in C^{\infty}(M)$ with average value $0$, then the distribution $w = \Pi \pi_0^* g$ is invariant, satisfying $Xw = 0$, and has pushforward $\pi_{0*}w = f$. In this case, we may use the Theorem~\ref{thm:inv-formula} inversion formula to find an explicit form for $g$, and hence $w$, in terms of $f$. Using results on the normal operator provided in \cite{2021-gouzel-lefeuvre}, we can extend this argument to any $f \in C^{\infty}(M)$
which is the content of Corollary~\ref{cor:inv-dist}. Such invariant distributions have been of recent interest \cite{2014-paternain-salo-uhlmann-spectral, 2015-paternain-salo-uhlmann-invariant, guillarmou2017invariant} due to their relationship to the X-ray transform and application to inverse problems.

Notice the presence of the Laplacian in the Theorem~\ref{thm:inv-formula} formula aligns with the observation of Guillarmou and Monard that $\Pi_0$ is a function of the Laplacian using representation theory \cite[Remark A.2]{guillarmou2017reconstruction}. There is likely a representation theory approach to proving this inversion formula, but we favor a more geometric approach, using Helgason's theory to work on the Poincar\'e disk.

There is a long history of X-ray inversion formulas \cite{radon1917determination, pestov2004characterization, guillarmou2017reconstruction} in the case of the X-ray transform on surfaces with boundary and Theorem~\ref{thm:inv-formula} is a first step to obtaining such an inversion formula in the Anosov case. Note explicit inversion formulas for the X-ray transform in the boundary case are only known for constant curvature surfaces (\cite{pestov2004characterization} for simple surfaces and \cite{guillarmou2017reconstruction} for surfaces with hyperbolic trapped set), but these formulas extend to approximate inversion formulas in the case of variable curvature. A natural question is if Theorem~\ref{thm:inv-formula} can be extended to an approximate inversion formula over Anosov surfaces of variable curvature, preferably with an error term related to the error in the boundary case.

%This X-ray transform on Anosov manifolds is analogous to the X-ray transform on manifolds with boundary defined by integrating over geodesics between the boundary and the $\Pi_0$ operator is analogous to the X-ray normal operator in this boundary case. There is a long history of X-ray reconstruction formulas on surfaces with boundary, beginning with Radon \cite{radon1917determination} on the Euclidean disk, then Pestov--Uhlmann \cite{pestov2004characterization} provided an approximate inversion formula on simple manifolds which is exact for constant curvature 
%(and Krishman \cite{krishnan2010inversion} showed can be made exact for small $C^3$ neighborhoods of constant curvature metrics), 
%which was extended by Guillarmou--Monard \cite{guillarmou2017reconstruction} to manifolds with hyperbolic trapped set. An interesting question is if Theorem~\ref{thm:inv-formula} can be extended to an approximate inversion formula for Anosov surfaces of variable curvature as in the Pestov--Uhlmann reconstruction.

Theorem~\ref{thm:inv-formula} generalizes an X-ray reconstruction formula of Helgason \cite{helgason1959differential} for the Poincar\'e disk. We prove this by first inverting an ``attenuated normal operator'' $\Pi_0^{(z)} \defeq \pi_{0*}\Pi^{(z)}\pi_0^*$ where the ``attenuated $\Pi$ operator'' $\Pi^{(z)}: C^{\infty}(SM) \to C^{\infty}(SM)$ is defined by
\begin{equation}
	\Pi^{(z)} \defeq \int_{\R} e^{-|t|z} \varphi_t^* dt, \quad \re z > 0.
	\label{eq:Piz}
\end{equation}
Guillarmou observed \cite{guillarmou2017invariant} this operator converges weakly to the $\Pi$ operator as the attenuation coefficient $z$ approaches $0$ when acting on functions with average value zero. That is, over an Anosov manifold $M$ with $f \in C^{\infty}(SM)$ we have
\begin{equation}
	\lim_{z \to 0} \ip{\Pi^{(z)}f}{\psi} = \ip{\Pi f}{\psi}, \quad \psi \in C^{\infty}(SM)
	\label{eq:Piz-Pi}
\end{equation}
provided $\int_{SM} f d\Sigma = 0$ for the Sasaki volume form $d\Sigma$ as defined in Section~\ref{sec:SM}. The distributional pairing $\ip{\cdot}{\cdot}$ above is also with respect to the Sasaki volume form. To invert this $\Pi_0^{(z)}$ operator, we extend the operator $S_K$ of (\ref{eq:S}) to the operator
\begin{equation}
	S_{K}^{(z)}f \defeq \int_{S_xM}\int_{0}^{\infty}e^{-(z+\sqrt{-K}) \cdot t}f(\gamma_{x,v}(t))dtdS_x(v),
	\quad f \in C^{\infty}(M).
	\label{eq:Sz}
\end{equation}
Furthermore, we will write $S^{(z)} = S_{-1}^{(z)}$ and $S = S_{-1}$. Before inverting $\Pi_0^{(z)}$ over closed hyperbolic surfaces, we first obtain the following inversion formula on the Poincar\'e disk, which is interesting in its own right and can be interpreted as a reconstruction formula for the ``attenuated X-ray transform'' operator $\Pi^{(z)}\pi^{*}_0$.
\begin{theorem}
	Given a compactly supported smooth function $f$ over the Poincar\'e disk $\D$, for any $\re(z) > 0$ we have the inversion formula
	\begin{equation}
		(\Delta - z(z+1)) S^{(z)} \Pi^{(z)}_0 f = -8\pi^2 f.
		\label{eq:D-att-inv}
	\end{equation}
	\label{thm:D-att-inv}
\end{theorem}
%pretty sure the theorem this is extending is originally due to Helgason
The above theorem and corresponding proof is an extension of Helgason's inversion formula for the unattenuated X-ray transform on the Poincar\'e disk \cite[Theorem 1.14, Chapter III]{helgason1980radon}. By descending to quotients of the Poincar\'e disk while normalizing with respect to curvature, we obtain this inversion formula on closed manifolds of constant negative curvature. The theorem below is an intermediate step to proving Theorem~\ref{thm:inv-formula}, but can be interpreted as a reconstruction formula for the attenuated X-ray transform operator $\Pi^{(z)}\pi^*_0$.
\begin{theorem}
	Given a smooth function $f$ over a connected closed surface $M$ of constant curvature $K < 0$, then for $\re(z) > 0$ we have the inversion formula
	\begin{equation}
		(\Delta - z(z+\sqrt{-K})) S_K^{(z)} \Pi^{(z)}_0 f = -8\pi^2 f.
		\label{eq:M-att-inv}
	\end{equation}
	\label{thm:M-att-inv}
\end{theorem}
Note the lift of a nonzero function to the Poincar\'e disk will not be compactly supported, so Theorem~\ref{thm:D-att-inv} does apply directly and we must construct some bounds to control the contribution of the function far from the center of the disk. With the above theorem, we obtain Theorem~\ref{thm:inv-formula} by taking $z \to 0$, but obtaining continuity of the operator $\Pi_0^{(z)}$ at $0$ to make such an argument is non-trivial as this requires sufficiently chaotic dynamics of the geodesic flow. We make use of microlocal techniques together with Klingenberg's transversality result to obtain this continuity.
%
%It would be useful to find a proof of Theorem~\ref{thm:inv-formula} without taking such a limit, for perhaps such a proof would have a better chance to generalize to an (approximate) inversion formula for variable curvature. 
Throughout the paper we will use the notation
\begin{equation}
	L_K^{(z)} \defeq \Delta - z(z+\sqrt{-K})
\end{equation}
as this operator appears in Theorem~\ref{thm:M-att-inv} and we will write $L^{(z)} = L_{-1}^{(z)}$ to denote the operator appearing in Theorem~\ref{thm:D-att-inv}. For future use, note the definitions of the pullback $\pi_0^*$ and the pushforward $\pi_{0*}$ gives:
\begin{equation}
	\Pi_0^{(z)}f = \int_{S_xM}\int_{\R}e^{-z|t|}f(\gamma_{x,v}(t))dtdS_x(v),
	\label{eq:Piz-int}
\end{equation}
which provides the useful relation $S_K^{(z)} = \f{1}{2}\Pi_0^{(z+\sqrt{-K})}$.

%TODO somewhere in the above mention representation theory relationship 

%TODO? some open questions...
% non-const curv (perhaps an approximate inversion formula as in the simple case)
% see if techniques can be applied to simple manifolds
% consisder non constant attenuation coefficient a and define Pi^{(a)}?
% can this be proven more directly (i.e. without Pi^{(z)})?
% inversion formula for I_m.

%TODO? perhaps don't mention psi-do property and injectivity...

\section{Preliminaries}
\label{sec:pre}

\subsection{Unit tangent bundle and geodesic flow}
\label{sec:SM}
For any Riemannian manifold $(M,g)$, the \emph{unit tangent bundle} is defined $SM = \{(x,v) \in TM : |v|_g = 1\}$ and comes with natural projection $\pi_0: SM \to M$. The unit tangent bundle is the phase space for the \emph{geodesic flow} $\varphi_t: SM \to SM$ given by $\varphi_t(x,v) = (\gamma_{x,v}(t), \dot{\gamma}_{x,v}(t))$ and has infinitesimal generator $X$, a smooth vector field on $SM$. Let $\X \defeq \R X$ denote the subbundle of $TSM$ spanned by $X$,  let  $\V \defeq \ker d\pi_0$ denote the \emph{vertical subbundle}, and take the \emph{horizontal subbundle} $\H$ to consist of the velocity vectors of all curves $(\gamma(t), W(t)) \in SM$ such that $\nabla_{\dot{\gamma}(t)}W = 0$ where $\nabla$ is the Levi-Civita connection. This provides the splitting
\begin{equation*}
	TSM = \X \oplus \H \oplus \V.
\end{equation*}
For surfaces, these subbundles are spanned by vector fields $X$, $H$, and $V$ respectively, normalized to satisfy the \emph{structure equations} $[X,V] = H$, $[H,V] = -X$, and $[X,H] = -KV$ where $K$ is the Gaussian curvature. Taking $\{X,H,V\}$ to be an orthonormal frame defines the \emph{Sasaki metric} on $SM$, which induces induces volume form $d\Sigma$ on $SM$ and volume form $dS_x$ on $S_xM$ -- see \cite{paternain2023geometric} for more details.

\subsection{Anosov manifolds}
\label{sec:anosov}
The theory in defining Guillarmou's operator $\Pi$ requires the dynamics of the geodesic flow to be sufficiently chaotic. In particular, we require $\varphi_t$ to be \emph{Anosov}, meaning there is a continuous and flow-invariant splitting of the tangent space
\begin{equation*}
	TSM = \R X \oplus E_s \oplus E_u
\end{equation*}
into the flow direction, the \emph{stable bundle}, and the \emph{unstable bundle} respectively such that given an arbitrary metric $|\cdot|$ on $SM$ there exists constants $C, \lambda > 0$ so that for all $t \geq 0$
\begin{align*}
	|d\varphi_t(v)| &\leq Ce^{-\lambda t}|v| \quad \text{for } v \in E_s,\\
	|d\varphi_{-t}(v)| &\leq Ce^{-\lambda t}|v| \quad \text{for } v \in E_u.
\end{align*}
A Riemannian manifold with Anosov geodesic flow is called an \emph{Anosov manifold}, which, for example, includes all Riemannian manifolds with strictly negative sectional curvature \cite{anosov1969geodesic}.
In applying microlocal tools, we are more interested in the dual splitting
\begin{equation*}
	T^*SM = E^*_0 \oplus E^*_s \oplus E^*_u
\end{equation*}
defined so that $E^*_0(E_s \oplus E_u) = 0$, $E^*_s(E_s \oplus \R X) = 0$, and $E^*_u(E_u \oplus \R X) = 0$.
The microlocal tools we will use are the \emph{wavefront set} which we denote by $\WF$ together with \emph{pseudo-differential operators} of order $m$ over $SM$ which we denote $\Psi^m(SM)$ and have corresponding \emph{symbol class} $S^m(T^*SM)$.

%TODO? perhaps more...

\subsection{Meromorphic extension of resolvents}
\label{sec:mer-ext}
\label{sec:pi}
%TODO? more to cite here?
We can decompose the operator $\Pi^{(z)}$ defined in (\ref{eq:Piz}) by
\begin{equation*}
	\Pi^{(z)}
	= \int_0^{\infty} e^{-tz} \varphi^*_{-t} dt
	+ \int_0^{\infty} e^{-tz} \varphi^*_{t} dt
	= R_+(z) + R_-(z).
\end{equation*}
where $R_+(z)$ and $R_-(z)$ denote the operators given by the left and right integrals respectively. A direct computation reveals $(z+X)R_+(z) = \id = (z-X)R_-(z)$, meaning these operators are the resolvents $R_{\pm}(z) = (z \pm X)^{-1}$. Thus the problem of defining $\Pi$ by  extending $\Pi^{(z)}$ to $z = 0$ as in (\ref{eq:Piz-Pi}) is the problem of extending the resolvents $R_{\pm}(z)$ to $0$. 

By the result of Faure--Sj\"ostrand \cite{faure2011upper}, there exist certain anisotropic Sobolev spaces $H^s_{\pm}$ over $SM$ so that the resolvent mappings $R_{\pm}(z): H_{\pm}^s \to H_{\pm}^s$  extend meromorphically to $0$ over these spaces for all $s > 0$.
%TODO? briefly recall this construction as this is more Thibault's notation
The anisotropic Sobolev space $H^s_+$ (resp. $H^s_-$) is defined microlocally to have Sobolev regularity $H^s$ (resp. $H^{-s}$) in a conic neighborhood $V_s$ of $E_s^*$ and regularity $H^{-s}$ (resp. $H^s$) in a conic neighborhood $V_u$ of $E_u^*$. In fact, we may choose $V_{s/u} \supset E_{s/u}^*$ to be any disjoint conic neighborhoods \cite{lefeurve-microlocal-dynamics}.
It follows that for any pseudodifferential operator ${B \in \Psi^0(SM)}$ with $\WF(B) \subset V_{s}$, we find $B: H^s_+ \to H^s(SM)$ is bounded and in general, this construction provides the bounded inclusions ${H^{s}(SM) \inj H^{s}_{\pm} \inj H^{-s}(SM)}$. Dyatlov-Zworski \cite{dyatlov2016dynamical} studied the wavefront set of these resolvents, which is key input to Guillarmou's \cite{guillarmou2017invariant} analysis on the case of geodesic flows over an Anosov manifold. In this case, the extensions $R_{\pm}(z): H_{\pm}^s \to H_{\pm}^s$ have no poles for $\re z > 0$ and have Laurent expansion near $0$ given by
\begin{equation*}
	R_{\pm}(z) = R_{\pm}^{\hol}(z) + \f{\Pi^{\pm}}{z}
\end{equation*}
where the spectral projections $\Pi^{\pm} = \pm 1 \otimes 1$ are given by the distributional pairings $\ip{\Pi^{\pm}u}{v} = \pm\ip{u}{1}\ip{v}{1}$. That is, the resolvents have a simple pole at $0$ and if $\int_{SM} f d\Sigma = 0$, then $R_{\pm}(z)f$ extends \emph{holomorphically} to $0$. 
Guillarmou's operator $\Pi: H^{s}(SM) \to H^{-s}(SM)$ is defined by
\begin{equation}
\Pi \defeq R^{\hol}_+(0) + R^{\hol}_-(0)
\end{equation}
and has notable properties $X\Pi = 0$ and $\Pi 1 = 0$. Guillarmou's normal operator is $\Pi_0 \defeq \pi_{0*} \Pi \pi_0^*$, which is an elliptic pseudodifferential operator of order $-1$. See \cite{guillarmou2017invariant} for more details.

\subsection{The Poincar\'e disk} 
\label{sec:poincare-disk}
We denote by $\D = \{z \in \C : |z| < 1\}$ the \emph{Poincar\'e disk} with metric
\begin{equation*}
	g = \f{4|dz|^2}{(1-|z|^2)^2}
\end{equation*}
which has constant Gaussian curvature $-1$.
%
%GROUP THEORY SETUP
%
The isometry group of the Poinar\'e disk is the matrix group $G = SU(1,1)$ given by %\todo{give more intrinisic def?}
\[
	G = \l\{
	\begin{pmatrix}
		a & b\\
		\ol{b} & \ol{a}
	\end{pmatrix}
	:
	|a|^2-|b|^2 = 1
	\r\}
\]
where $G$ acts transitively and isometrically on the Poincar\'e disk by the action
\[
	\begin{pmatrix}
		a & b\\
		\ol{b} & \ol{a}
	\end{pmatrix}
	\cdot
	z
	= \f{az+b}{\ol{b}z+\ol{a}}.
\]
Let $K < G$ be the subgroup of $G$ that stabilizes the origin $0 \in \D$. Then if $X = G/K$ we obtain the diffeomorphism $X \cong \D$, so $\D$ is a homogeneous space, and for the proof of Theorem~\ref{thm:D-att-inv} we will need the following theory on homogeneous spaces -- more details can be found in \cite{helgason1980radon, helgason2022groups}. 
%
%CONVOLUTION
%
The group structure of homogeneous spaces allows for the following notion of convolution, which will provide useful representations for the key operators $S^{(z)}$ and $\Pi_0^{(z)}$ in the inversion formula. Given two functions $f_1, f_2$ over $\D$, their convolution is defined as
\begin{equation}
	(f_1 \times f_2)(g \cdot 0) 
	:= \int_G f_1(h \cdot 0)f_2(h^{-1}g \cdot 0)dh  
	= \int_G f_1(gh^{-1} \cdot 0)f_2(h \cdot 0)dh.
	\label{eq:conv}
\end{equation}
for some fixed $g \in G$. Here $dh$ is the Haar measure  of $G$ normalized so that for any $f \in C_c(\D)$ we have $\int_{G} f(h \cdot 0) dh = \int_{\D} f(x) d\vol_{\D}(x)$ where $d\vol_{\D}$ is the Riemannian volume form on the Poincar\'e disk. Note in our case $G = SU(1,1)$ is unimodular, which allows us to conflate the left and right Haar measures and obtain the equality in (\ref{eq:conv}); also note it is straightforward to verify this convolution is associative. Furthermore,  we call a function $f$ on $\D$ \emph{radial} if it is $K$-invariant, meaning $f(k \cdot x) = f(x)$ for all $k \in K$ and $x \in X$, and it is a theorem that this convolution is commutative when $f_1$ and $f_2$ are radial \cite[Corollary 5.2, Chapter II]{helgason2022groups}.
%laplace-beltrami and convolution
Using the definition of convolution and that the Laplace-Beltrami operator $\Delta$ is invariant under the isometry group $G$, we find $\Delta$ distributes over convolution by
\begin{equation}
	\Delta(f_1 \times f_2) = f_1 \times (\Delta f_2).
	\label{eq:delta-conv}
\end{equation}
Importantly, in the case both $f_1$ and $f_2$ are radial, we have commutativity of convolution and so we may distribute $\Delta$ and $L^{(z)}$ into the first factor.
%
%SYMMETRIZATION
%
We will be able to leverage the utility of radial functions by using the radial symmetrization operation
\begin{equation*}
f^{\natural}(x) = \int_{K}f(k \cdot x)dk.
\end{equation*}
In particular, we will be able to reduce our problem to radial functions by using
\begin{equation}
	(f_1 \times f_2)^{\natural} = f_1^{\natural} \times f_2.
	\label{eq:conv-nat}
\end{equation}
That is, the symmetrization operator distributes to the first convolution factor. Furthermore, by the isometry invariance of $\Delta$ and that $K$ stabilizes $0$, we know
\begin{equation}
	(\Delta f)^{\natural}(0) = (\Delta f^{\natural})(0).
	\label{eq:delta-nat}
\end{equation}
%
%SPHERICAL TRANSFORM
%
The key in verifying the inversion formula of the Theorem~\ref{thm:D-att-inv} is showing the spherical transform of each side agrees. To define this transform, first recall a \emph{spherical function} $\varphi$ on $\D$ is by definition a radial eigenfunction of the Laplace-Beltrami operator satisfying $\varphi(0) = 1$. We denote by $\varphi_{\lambda}$ the spherical function with eigenvalue $-(\lambda^2 + \f{1}{4})$ which has explicit integral representation
\begin{equation*}
	\varphi_{\lambda}(r) = \f{1}{\pi}\int_0^{\pi}(\cosh r - \sinh r \cos\theta)^{-i\lambda + 1/2}d\theta.
\end{equation*}
Then the \emph{spherical transform} of a radial function $f \in C_c(\D)$ is defined by
\begin{equation*}
	\til{f}(\lambda) = \int_{\D}f(x)\varphi_{-\lambda}(x)dx
\end{equation*}
for all $\lambda \in \C$ for which this integral exists. We need a few key properties of this transform. Firstly, this spherical transform is injective as there is a known inversion formula. Secondly, the spherical transform intertwines with the Laplacian by
\begin{equation}
	(\Delta f)^{\sim}(\lambda) = -\l(\lambda^2 + \f{1}{4}\r)\til{f}(\lambda)
	\label{eq:L-ST}
\end{equation}
and the spherical transform of a convolution is the product of the transforms:
\begin{equation}
	(f_1 \times f_2)^{\sim}(\lambda) = \til{f}_1(\lambda)\til{f}_2(\lambda).
	\label{eq:conv-ST}
\end{equation}
Finally, rephrasing \cite[Lemma 1.12, Chapter III]{helgason1980radon} yields
\begin{equation}
	2\pi\int_0^{\infty}e^{-zr}\varphi_{-\lambda}(r)dr
	= \pi\f{
	\Gamma(\f{z}{2} + \f{1}{4} + \f{\lambda}{2}i)\Gamma(\f{z}{2} + \f{1}{4} - \f{\lambda}{2}i)}
	{\Gamma(\f{z}{2} + \f{3}{4} + \f{\lambda}{2}i)\Gamma(\f{z}{2} + \f{3}{4} - \f{\lambda}{2}i)}.
	\label{eq:gamma-int}
\end{equation}

\section{Attenuated inversion on the Poincar\'e disk}
In this section we prove the attenuated inversion formula on the Poincar\'e disk $\D$ described in Theorem~\ref{thm:D-att-inv} by generalizing the argument for an unattenuated X-ray inversion formula given in \cite[Theorem 1.14, Chapter III]{helgason1980radon}. This argument requires the preliminaries and notation of Section~\ref{sec:poincare-disk} on the Poincar\'e disk. First we write
$\Pi_0^{(z)}f$ and $S^{(z)}f$ as a convolution.
%\begin{equation}
%\int_G f(h) dh = \int_{G/K} \int_K f(ak) dk d\vol_{\D}(aK)
%\label{eq:normalization}
%\end{equation}
%\cite[Theorem 2.51]{folland2016course}
\begin{lemma}
\label{lem:op-conv}
	We have $\Pi_0^{(z)}f = f \times 2\tau^{(z)}$ and $S^{(z)}f = f \times \sigma^{(z)}$
	with kernels
	\begin{align*}
	\tau^{(z)}(x) &= e^{-zd(0,x)}\sinh d(0,x)^{-1}\\
	\text{and} \quad
	\sigma^{(z)}(x) &= e^{-(z+1)d(0,x)}\sinh d(0,x)^{-1}
	= e^{-zd(0,x)}(\coth d(0,x) - 1).
	\end{align*}
\end{lemma}
\begin{proof}
First note that if $f_2$ is radial (as $\tau^{(z)}$ and $\sigma^{(z)}$ are), we can use the normalization of the convolution and that $K$ stabilizes $0$ to write
\begin{align*}
	(f_1 \times f_2)(g \cdot 0)
	= \int_G f_1(h \cdot 0)f_2(h^{-1}g \cdot 0)dh
%	&= \int_{G/K} \int_K f_1(ak \cdot 0)f_2(k^{-1}a^{-1}g) dk d\vol_{\D}(aK)\\
	= \int_{\D} f_1(a \cdot 0)f_2(a^{-1}g) d\vol_{\D}(aK).
%	&= \int_{\D} f_1(x) F_2(d(F_2(0,a^{-1}g \cdot 0)))d\vol_{\D}(aK)\\
%	&= \int_{\D} f_1(x) F_2(d(F_2(x,g \cdot 0)))d\vol_{\D}(aK)\\
\end{align*}
Thus if we write $f_2(x) = F_2(d(0,x))$, we can express the convolution as
\begin{align*}
	(f_1 \times f_2)(x)
	= \int_{\D} f_1(x)F_2(d(x,y)) d\vol_{\D}(y).
\end{align*}
With this, use the substitution $y = \gamma_{x,v}(t)$ to rewrite the operator $\Pi_0^{(z)}$ as
\begin{align*}
	\Pi_0^{(z)}f(x) &= 2\int_{S_x\D}\int_{0}^{\infty} e^{-zt}f(\gamma_{x,v}(t))dtdS_x(v)\\
	&= 2\int_{\D}e^{-zd(x,y)}\sinh d(x,y)^{-1} f(y)dy
%	&= \int_{\D}e^{-zd(x,y)}(\coth d(x,y) - 1)f(y)dy
	= f \ast 2\tau^{(z)}
\end{align*}
where we used $dy = \sinh d(x,y)dtdS_x(v)$ on the Poincar\'e disk. The expression for $S^{(z)}f$ follows from $S^{(z)} = \f{1}{2}\Pi_0^{(z+1)}$.
\end{proof}

This lemma provides access to the theory of homogeneous spaces. For example, it is an immediate consequence of Lemma~\ref{lem:op-conv} and (\ref{eq:conv-nat}) that
\begin{equation}
	(\Pi_0^{(z)}f)^{\natural} = \Pi_0^{(z)}f^{\natural}
	\quad \text{and} \quad
	(S^{(z)}f)^{\natural} = S^{(z)}f^{\natural}.
	\label{eq:S-Pi-nat}
\end{equation}
We additionally note it is an immediate consequence of (\ref{eq:delta-conv}) that the operator $L^{(z)} = \Delta - z(z+1)$ also distributes over convolution by
\begin{equation}
	L^{(z)}(f_1 \times f_2) = f_1 \times (L^{(z)} f_2)
	\label{eq:Lz-conv}
\end{equation}
and it is a consequence of (\ref{eq:delta-nat}) that
\begin{equation}
	(L^{(z)}f)^{\natural}(0) = (L^{(z)}f^{\natural})(0).
	\label{eq:Lz-nat}
\end{equation}

\begin{proof}[Proof of Theorem~\ref{thm:D-att-inv}]
	It suffices to prove the inversion formula at the origin $0 \in \D$ as all operators involved ($L^{(z)}$, $S^{(z)}$, $\Pi_0^{(z)}$) are invariant under the isometry group $G$. To reduce the problem to radial functions $f$, we use (\ref{eq:Lz-nat}) to rewrite the left side of the inversion formula as
	\begin{align*}
		(L^{(z)}S^{(z)}\Pi_{0}^{(z)}f)(0)
		= (L^{(z)}(S^{(z)}\Pi_0^{(z)}f)^{\natural})(0)
		= (L^{(z)}S^{(z)}\Pi_0^{(z)}f^{\natural})(0).
	\end{align*}
	Note we used (\ref{eq:S-Pi-nat}) to commute the radial symmetrization with the operators in the last equality. This equation together with the observation $-8\pi^2 f^{\natural}(0) = -8\pi^2 f(0)$ implies it suffices to prove the theorem for the radial symmetrization of $f$. That is, we can assume without loss of generality that $f$ is radial and so it suffices to verify
	\begin{equation*}
		L^{(z)}S^{(z)}\Pi_0^{(z)}f(0) = -8\pi^2f(0)
		\quad \text{for } f \in C_c^{\infty}(\D) \text{ radial}.
	\end{equation*}
	We prove this by taking a spherical transform of both sides and ensuring these transforms agree. Using (\ref{eq:Lz-conv}) with convolution is commutative for radial functions, followed by (\ref{eq:conv-ST}), we compute the spherical transform of the left side to be
	\begin{align*}
		(L^{(z)}S^{(z)}\Pi_0^{(z)}f)^{\sim}(\lambda)
		&= 2(L^{(z)}(f \times \tau^{(z)} \times \sigma^{(z)}))^{\sim}(\lambda)\\
		&= 2(L^{(z)}f \times \tau^{(z)} \times \sigma^{(z)})^{\sim}(\lambda)
		= 2(L^{(z)}f)^{\sim}(\lambda)\til{\tau}^{(z)}(\lambda)\til{\sigma}^{(z)}(\lambda).
	\end{align*}
	We are done if we can reduce the above expression to $-8\pi^2\til{f}(\lambda)$. First use $L^{(z)} = \Delta - z(z+1)$ and (\ref{eq:L-ST}) to write
	\begin{align*}
		(L^{(z)})^{\sim}(\lambda)
		&= (\Delta f)^{\sim}(\lambda) - z(z+1)\til{f}(\lambda)\\
		&= \l(-\l(\lambda^2+\f{1}{4}\r)-z(z+1)\r)\til{f}(\lambda)
		= -\l(\l(z+\f{1}{2}\r)^2 + \lambda^2\r)\til{f}(\lambda).
	\end{align*}
	Thus it only remains to compute the spherical transforms of $\tau^{(z)}$ and $\sigma^{(z)}$ and show 
	\begin{equation}
		\til{\tau}^{(z)}(\lambda)\til{\sigma}^{(z)}(\lambda)
		 = \f{4\pi^2}{\l(z+\f{1}{2}\r)^2 + \lambda^2}.
		\label{eq:sigma-times-tau}
	\end{equation}
	Rewriting the spherical transform integral of $\tau^{(z)}$ in polar coordinates yields
	\begin{align*}
	\til{\tau}^{(z)}(\lambda) 
	&= \int_{\D}e^{-zd(0,x)}\sinh d(0,x)^{-1} \varphi_{-\lambda}(d(0,x))dx\\
	&= \int_{0}^{2\pi}\int_{0}^{\infty}e^{-zr}(\sinh r)^{-1}\varphi_{-\lambda}(r)\sinh r drd\theta
	= 2\pi \int_{0}^{\infty}e^{-zr}\varphi_{-\lambda}(r)dr.
	\end{align*}
	The above integral is precisely (\ref{eq:gamma-int}) and so we have computed
	\begin{equation}
	\til{\tau}^{(z)}(\lambda)		
	= \pi\f{
	\Gamma(\f{z}{2} + \f{1}{4} + \f{\lambda}{2}i)\Gamma(\f{z}{2} + \f{1}{4} - \f{\lambda}{2}i)}
	{\Gamma(\f{z}{2} + \f{3}{4} + \f{\lambda}{2}i)\Gamma(\f{z}{2} + \f{3}{4} - \f{\lambda}{2}i)}.
	\label{eq:tau-ST}
	\end{equation}
	Next we compute the spherical transform $\til{\sigma}^{(z)}(\lambda)$, which is precisely designed to cancel the above gamma functions. Indeed, by substituting polar coordinates in the same way as above we find
	\begin{align*}
		\til{\sigma}^{(z)}(\lambda)
		= 2\pi \int_{0}^{\infty}e^{-(z+1)r}\varphi_{-\lambda}(r)dr.
	\end{align*}
	Evaluating this integral using (\ref{eq:gamma-int}), then applying $\Gamma(w+1) = w\Gamma(w)$ to both factors in the denominator gives
	\begin{equation}
	\til{\sigma}^{(z)}(\lambda) 
	= \f{\pi}{\l(\f{z}{2}+\f{1}{4}\r)^2 + \l(\f{\lambda}{2}\r)^2}
	\cdot 
	\f{\Gamma(\f{z}{2} + \f{3}{4} + \f{\lambda}{2}i)\Gamma(\f{z}{2} + \f{3}{4} - \f{\lambda}{2}i)}
	{\Gamma(\f{z}{2} + \f{1}{4} + \f{\lambda}{2}i)\Gamma(\f{z}{2} + \f{1}{4} - \f{\lambda}{2}i)}.
	\label{eq:sigma-ST}
	\end{equation}
	We then obtain (\ref{eq:sigma-times-tau}) by multiplying (\ref{eq:tau-ST}) with (\ref{eq:sigma-ST}).
\end{proof}

\section{Attenuated inversion on constant curvature Anosov surfaces}

In this section we prove Theorem~\ref{thm:M-att-inv}, which is the attenuated inversion formula on a connected closed surface $M$ with a metric $g$ of constant negative curvature. We would like to build upon Theorem~\ref{thm:D-att-inv}, which assumes constant Gaussian curvature $K = -1$, so we must apply the following normalization process.

\begin{lemma}
Let $(M,g)$ be a surface of constant curvature $K < 0$, let $\til{g} = -K g$ be the rescaled metric with constant curvature $-1$, and let $\Pi^{(z)}_0$, $S_K^{(z)}$, $L_K^{(z)}$ and $\til{\Pi}^{(z)}_0$, $\til{S}_{-1}^{(z)}$, $\til{L}_{-1}^{(z)}$ be the operators defined with metrics $g$ and $\til{g}$ respectively. Then,
\begin{equation*}
	L_K^{(z)}S_K^{(z)}\Pi_0^{(z)} = \til{L}_{-1}^{(z/\sqrt{-K})}\til{S}_{-1}^{(z/\sqrt{-K})}\til{\Pi}_0^{(z/\sqrt{-K})}.
\end{equation*}
\label{lem:norm}
\end{lemma}

\begin{proof}
	We relate the operators defined using $g$ to the corresponding operators using $\til{g}$, which is essentially an exercise in changing units. Consider a test function $f \in C_c^{\infty}(M)$ and use (\ref{eq:Piz-int}) to write 
	\begin{align*}
		\Pi_0^{(z)}f 
		= \int_{S_xM}\int_{\R} e^{-z|t|}f(\gamma_{x,v/\sqrt{-K}}(\sqrt{-K}t))dtdS_x(v).
	\end{align*}
	Then making the change of variables $\sqrt{-K}t \mapsto t$ and letting $\til{S}M$ be the unit sphere bundle with respect to $\til{g}$, we can write
	\begin{align*}
		\Pi_0^{(z)}f 
		= \f{1}{\sqrt{-K}} \int_{\til{S}_xM}\int_{\R} e^{-z/\sqrt{-K} \cdot |t|}f(\gamma_{x,\til{v}}(t))dtd\til{S}_x(\til{v})
		= \f{1}{\sqrt{-K}}\til{\Pi}_0^{(z/\sqrt{-K})}f.
	\end{align*}
	We can repeat the same argument for $S_K^{(z)}$, or simply observe
	\begin{align*}
		S_K^{(z)} 
		= \f{1}{2}\Pi_0^{(z+\sqrt{-K})}
		= \f{1}{2\sqrt{-K}} \til{\Pi}_0^{(z/\sqrt{-K}+1)}
		= \f{1}{\sqrt{-K}}\til{\Pi}_0^{(z/\sqrt{-K}+1)}
		= \f{1}{\sqrt{-K}}\til{S}_{-1}^{(z/\sqrt{-K})}.
	\end{align*}
	Next let $\Delta$ and $\til{\Delta}$ be the Laplace-Beltrami operators for $g$ and $\til{g}$ respectively, note $\Delta = -K\til{\Delta}$, then compute
	\begin{align*}
		L_K^{(z)}
		= \Delta - z(z+\sqrt{-K})
%		= -K\l(\f{1}{-K}\Delta - \f{z}{\sqrt{-K}}\l(\f{z}{\sqrt{-K}}+1\r)\r)
		= -K\l(\til{\Delta} - \f{z}{\sqrt{-K}}\l(\f{z}{\sqrt{-K}}+1\r)\r)
		= -K\til{L}_{-1}^{(z/\sqrt{-K})}.
	\end{align*}
	Therefore, using all the above operator relations, we conclude
	\begin{align*}
	L_K^{(z)}S_K^{(z)}\Pi_0^{(z)}
	&= -K\til{L}_{-1}^{(z/\sqrt{-K})}\f{1}{\sqrt{-K}}\til{S}_{-1}^{(z/\sqrt{-K})}\f{1}{\sqrt{-K}}\til{\Pi}_0^{(z/\sqrt{-K})}.\\
%	&= \til{L}^{(z/\sqrt{-K})}\til{S}^{(z/\sqrt{-K})}\til{\Pi}_0^{(z/\sqrt{-K})}
	\end{align*}	
\end{proof}

To prove Theorem~\ref{thm:M-att-inv} we will consider the Riemannian covering map $\pi: \D \to M$,
%promised by the Killing--Hopf theorem
which we will use to carry over the inversion formula on the Poincar\'e disk given in Theorem~\ref{thm:D-att-inv} to the case of closed manifolds. Crucially, the operators in the inversion formula commute with such covering maps.

\begin{lemma} Let $\pi: \til{M} \to M$ be a Riemannian covering map between arbitrary Riemannian manifolds. Then for $\re(z) > 0$,
	\begin{equation}
		\pi^{*}\Pi_0^{(z)} = \Pi_0^{(z)}\pi^{*}.
		\label{eq:pi-op-comm}
	\end{equation}
	\label{lem:pi-op-comm}
\end{lemma}
\begin{proof}
	Take arbitrary $f \in C_c^{\infty}(M)$ to test the claimed operator equality. The key is $\Pi^{(z)}$ is defined by integrating over geodesics and the projection of a geodesic under a Riemannian covering map remains a geodesic. Indeed, fix any $x \in \til{M}$ and additionally consider some $v \in S_x\til{M}$. Then the projection $\pi \circ \gamma_{x,v}$ must be a geodesic by $\pi$ a local isometry, but this geodesic must be $\gamma_{\pi(x), D_x\pi(v)}$ because both these geodesics have initial position $\pi(x)$ and initial velocity $D_x\pi(v)$.
%	\begin{align*}
%	 &(\pi \circ \gamma_{x,v})(0) = \pi(0) = \gamma_{\pi(x), D\pi(v)}(0)\\
%	 \text{and} \quad &(\pi \circ \gamma_{x,v})'(0) = D_x\pi(v) = \gamma_{\pi(x), D\pi(v)}'(0).
%	\end{align*}
Then by $\pi \circ \gamma_{x,v} = \gamma_{\pi(x), D_x\pi(v)}$ we get the relation
	\begin{align*}
		f \circ \gamma_{\pi(x), D\pi(v)}
		= f \circ \pi \circ \gamma_{x,v}
		= \pi^{*}f \circ \gamma_{x,v}.
	\end{align*}
This allows us to compute
	\begin{align*}
		(\Pi_0^{(z)}f)(\pi(x))
		&= \int_{S_xM}\int_{\R}e^{-|t|z}f(\gamma_{\pi(x), D\pi(v)}(t))dtdS_x(v)\\
		&= \int_{S_xM}\int_{\R}e^{-|t|z}\pi^{*}f(\gamma_{x,v}(t))dtdS_x(v)
		= (\Pi_0^{(z)}\pi^{*}f)(x).
	\end{align*}
\end{proof}

\begin{cor}
	Let $\pi: \til{M} \to M$ be a Riemannian covering map between arbitrary Riemannian manifolds. Then for $K < 0$ and $\re z \geq 0$,
	\begin{equation}
		\pi^{*}S_K^{(z)} = S_K^{(z)}\pi^{*}.
	\end{equation}
\end{cor}
\begin{proof}
	Using $S_K^{(z)} = \f{1}{2}\Pi_0^{(z+\sqrt{-K})}$, this follows directly from Lemma~\ref{lem:pi-op-comm}.
\end{proof}

Note the pullback of a smooth function into the Poincar\'e disk will almost certainly not be compactly supported, so to use Theorem~\ref{thm:D-att-inv} we will need the following lemma to control the effect of the $\Pi_0^{(z)}$ operator far from the origin on bounded functions.
\begin{lemma} Let $f \in C^{\infty}(\D)$ be uniformly bounded with $|f(y)| \leq C$ for all $y \in \D$ and fix $x \in \D$.
\begin{enumerate}[(a)]
\item Then $\Pi^{(z)}_0 f$ remains uniformly bounded, and $|\Pi_0 f (y)| \leq \f{4\pi}{\re(z)}C$ for all $y \in \D$.
\label{it:Pi0-bounded}
\item If $f(y) = 0$ for $d(x,y) < R$, then
	\begin{align*}
		|\Pi^{(z)}_0 f(x)| \leq \f{4\pi}{\re(z)}Ce^{-R\re(z)}.
	\end{align*}
	\label{it:C-0}
\item If $|f(y)| \leq \eps$ for all $y \in \D$ with $d(x,y) < R$, then
	\begin{align*}
		|\Pi^{(z)}_0 f(x)| \leq \f{4\pi}{\re(z)}(\eps + Ce^{-R\re(z)}).
	\end{align*}
	\label{it:C-eps}
\end{enumerate}
\label{lem:eps}
\end{lemma}
\begin{proof} Note that \ref{it:C-eps} implies \ref{it:C-0} by taking $\eps = 0$, and \ref{it:C-eps} implies \ref{it:Pi0-bounded} by letting $\eps = C$ and taking $R \to \infty$. Thus it only remains to show \ref{it:C-eps}. This bound follows quickly by breaking up the integral so that we consider the region close to $x$ and the region far from $x$ separately:
\begin{align*}
	|\Pi_0^{(z)}f(x)| &= 2\l|\int_{S_x\D}\int_{0}^{\infty}e^{-tz}f(\varphi_t(x,v))dtdS_x(v)\r|\\
%	&\leq 2\l|\int_{S_x\D}\int_{0}^{R}e^{-tz}f(\varphi_t(x,v))dtdS_x(v)\r|
%	  + 2\l|\int_{S_x\D}\int_{R}^{\infty}e^{-tz}f(\varphi_t(x,v))dtdS_x(v)\r|\\
	&\leq 2\eps\int_{S_x\D}\int_{0}^{R}|e^{-tz}|dtdS_x(v)
	  + 2C\int_{S_x\D}\int_{R}^{\infty}|e^{-tz}|dtdS_x(v)\\
%	&\leq 2 \eps \int_{S_x\D}\int_{0}^{\infty} |e^{-tz}| dtdS_x(v) + 2Ce^{-Rz}\int_{S_x\D}\int_{t=0}^{\infty}|e^{-tz}|dtdS_x(v)\\
	&\leq 2 \cdot 2\pi(\eps+Ce^{-Rz}) \int_{0}^{\infty} |e^{-tz}| dt
	\leq \f{4\pi}{\re(z)}(\eps + Ce^{-R\re(z)}).
\end{align*}
%TODO? the above computation should probably be broken up...
\end{proof}

%TODO? note that all operators are still defined, all operators can still be considered as convolutions, I can still take symmetrization, I can still commute symmetrization with convolutions, I still have commutativity for radial functions....
%TODO? (QUESTION) If possible, bound the Laplacian separately to avoid taking radial functions... techniques to do this?
%TODO? If I go with the above, I will probably need that Pi_0^{(z)} is Psi-do of order -1 before the below proof... or at least some mild mapping properties,
%TODO? Otherwise, spell out why such a chi that I need exists...
%TODO? Also otherwise, explain why Helgason's theorem extends to non-compactly supported
%TODO? possibly expand on argument for why it suffices to consider radial functions.
\begin{proof}[Proof of Theorem~\ref{thm:M-att-inv}]
By Lemma~\ref{lem:norm} we may assume without loss of generality that $M$ has constant curvature $-1$, and we verify the inversion formula in this case at an arbitrary $p \in M$ by lifting the problem to the Poincar\'e disk. Let $\pi: \D \to M$ be the universal Riemannian covering map and note, after composing with an isometry of the Poincar\'e disk if necessary, we may take $\pi(0) = p$. We will verify the Theorem~\ref{thm:M-att-inv} inversion formula by showing the lifts of both sides to the Poincar\'e disc are equal at this point; that is, we will show
\begin{equation*}
	\pi^{*}(L^{(z)}S^{(z)}\Pi_0^{(z)}f)(0) = -8\pi^2(\pi^{*}f)(0).
\end{equation*}
By Lemma~\ref{lem:pi-op-comm}, we may commute the pullback $\pi^{*}$ with these operators and so we only must verify the following equivalence on the Poincar\'e disk:
\begin{equation*}
	(L^{(z)}S^{(z)}\Pi_0^{(z)}\pi^*f)(0) = -8\pi^2\pi^{*}f(0).
\end{equation*}
Note the above equality nearly follows from Theorem~\ref{thm:D-att-inv}, but this theorem only holds for compactly supported functions, and $\pi^{*}f$ is almost certainly not compactly supported. To resolve this, we take a smooth radial cutoff function $\chi_R \leq 1$ with $\chi_R(x) = 1$ for $d(0,x) \leq R$ and $\chi_R(x) = 0$ for $d(0,x) \geq 2R$, so that for all $R$ we can write
\begin{align*}
		L^{(z)}S^{(z)}\Pi_0^{(z)}\pi^{*}f(0)
		&= L^{(z)}S^{(z)}\Pi_0^{(z)}(\chi_R \pi^{*}f + (1-\chi_R) \pi^{*}f)(0)\\
		&= -8\pi^2\pi^{*}f(0) + (L^{(z)}S^{(z)}\Pi_0^{(z)}(1-\chi_R)\pi^{*}f)(0).
\end{align*}
Note we applied Theorem~\ref{thm:D-att-inv} in the second equality, so now it only remains to show the last term goes to $0$ as $R \to \infty$. Note that by the same reasoning as in the proof of Theorem~\ref{thm:D-att-inv}, it suffices to assume $\pi^*f$ is radial, which simplifies the argument by allowing us to use commutativity of convolution and (\ref{eq:Lz-conv}) to commute $L^{(z)}$ with the other operators:
\begin{align*}
	(L^{(z)}S^{(z)}\Pi_0^{(z)}(1-\chi_R)\pi^{*}f)(0) 
	&= 2L^{(z)}((1-\chi_R)\pi^{*}f \ast \tau^{(z)} \ast \sigma^{(z)})(0)\\
	&= 2(L^{(z)}((1-\chi_R)\pi^{*}f) \ast \tau^{(z)} \ast \sigma^{(z)})(0)\\
	&= S^{(z)}\Pi_0^{(z)}L^{(z)}((1-\chi_R)\pi^{*}f)(0).
\end{align*}
Now define $h_R = L^{(z)}(1-\chi_R)\pi^{*}f$ and by the above computation, it suffices to show $S^{(z)}\Pi_0^{(z)}h_R(0) \to 0$ as $R \to \infty$. Furthermore, note $\pi^{*}f$ will have bounded derivatives, and we can take $\chi_R$ to have bounded first and second derivatives uniformly in $R$. That is, we can take $h_R \leq C$ to be uniformly bounded for all $R$, and note we will have $h_{R}(y) = 0$ for $d(0,y) \leq R$. Thus by Lemma~\ref{lem:eps} \ref{it:C-0}, we find that for any $x$ with $d(0,x) \leq R/2$ we have
\begin{align*}
	|\Pi_0^{(z)}h_R(x)| \leq \f{4\pi}{\re(z)} Ce^{-R/2 \re(z)}.
\end{align*}
Additionally, Lemma~\ref{lem:eps} \ref{it:Pi0-bounded} provides the uniform bound $|\Pi_0^{(z)}h_R(x)| \leq \f{4\pi}{\re(z)} C$ for all $x \in \D$. Recall $S^{(z)} = \f{1}{2}\Pi_0^{(z+1)}$ so Lemma~\ref{lem:eps} \ref{it:C-eps} with $\eps = \f{4\pi}{\re(z)} Ce^{-R/2 \re(z)}$ yields
\begin{align*}
	|S^{(z)}\Pi_0^{(z)}h_R(0)| 
	\leq \f{1}{2} \f{4\pi}{\re(z)} \cdot \f{4\pi}{\re(z+1)} C \l(e^{-R/2\re(z)} + e^{-R/2\re(z+1)}\r),
\end{align*}
which implies $|S^{(z)}\Pi_0^{(z)}h_R(0)| \to 0$ as $R \to \infty$.
\end{proof}

\section{Inversion of the normal operator over closed hyperbolic surfaces}

In this section we prove the inversion formula of Theorem~\ref{thm:inv-formula} by passing the Theorem~\ref{thm:M-att-inv} inversion formula to the limit $z \to 0$. The crux is showing continuity of the operator $\Pi_0^{(z)}$ in an appropriate sense.

\begin{lemma}
	On any Anosov manifold $M$ and $s > 0$, the operator $\Pi_0^{(z)}: H^s(M) \to H^s(M)$ is  meromorphic in $z$ and a bounded operator when defined.
\end{lemma}

%TODO? check pushforward and pullback are continuous in the way I need them...
%TODO? need to choose m(x,\xi) = 1 on X^*+H^*, which requires Thibault's "adapted function" remark in ch. 9. Does this result exist somewhere I can cite it? If not, is there a workaround? If I use this route, see my notes for what notation I should use here and in intro.
\begin{proof}
	First use $\Pi^{(z)} = R_+(z) + R_-(z)$ to decompose the operator as
	\begin{equation}
		\Pi_0^{(z)} 
		= \pi_{0*}\Pi^{(z)}\pi_0^* 
		= \pi_{0*}R_+(z)\pi_0^* + \pi_{0*}R_-(z)\pi_0^*.
		\label{eq:resolvent-decomp}
	\end{equation}
	We will show $\pi_{0*}R_{\pm}(z)\pi_0^*: H^s(M) \to H^s(M)$ is bounded and meromorphic. Indeed, note $\pi_0^{*}: H^s(M) \to H^s(SM)$ is bounded, and $R_+(z): H^s_+ \to H^s_+$ is bounded and meromorphic in $z$ between the anisotropic Sobolev spaces $H^s_+$ by construction, which when combined with the bounded inclusion $H^s(SM) \inj H^s_+$ gives us a bounded and meromorphic map $R_+(z)\pi_0^*: H^s(M) \to H^s_+$. Thus it only remains to show $\pi_{0*}: H^s_+ \to H^s(M)$ is bounded, which requires analyzing the wavefront set of $\pi_{0*}$. Given $u \in D'(SM)$, then by \cite[Proposition 11.3.3]{friedlander-distributions} we have:
	\begin{align*}
		\WF(\pi_{0*}u)
		\subset
		\{(x,\xi) \in T^*M : \exists v \in S_xM, ((x,v), d\pi_0^{\top}\xi, 0) \in \WF(u)\}.
	\end{align*}
	Note that by $\ker(d\pi_{0}) = \V$, we have $((x,v), d\pi_0^{\top}\xi, 0) \in \X^* \oplus \H^*$. We compare this wavefront to the unstable bundle $E^*_u$ for this is the direction $H^s_+$ has low microlocal regularity. Indeed, by definition we have $(\X^* \oplus \H^*)(\V) = 0$ and $E_u^*(E_u \oplus \X) = 0$, but by transversality \cite{1974-klingenberg-riemannian, 1987-mane-on-a-theorem-of-klingenberg} we have $\V \cap (E_u \oplus \X) = \{0\}$
and therefore $(\X^* \oplus \H^*) \cap E_u^* = \{0\}$, which is the key property to showing continuity of $\pi_{0*}$. To complete this claim formally, we use the notation of Section~\ref{sec:mer-ext} and choose the conic neighborhood $V_s$ of $E_s^*$ to contain $\X^* \oplus \H^*$. Next choose symbol $b \in S^0(T^*SM)$ so that $b(x,\xi) \equiv 1$ in a neighborhood of $\X^* \oplus \H^*$ and $b(x,\xi) = 0$ outside of $V_s$. Consider a quantization $B = \Op(b) \in \Psi^0(SM)$ and decompose
	\begin{align*}
		\pi_{0*} = \pi_{0*}B + \pi_{0*}(1-B).
	\end{align*}
Because $\WF(1-B)$ does not intersect $\X^* \oplus \H^*$, we conclude the term $\pi_{0*}(1-B)$ is smoothing and therefore is continuous as a map $H^{s}_+ \to H^s(SM)$.
%TODO? check this...
Next note that because $\WF(B)$ is contained in $V_s$, it follows from definition of anisotropic Sobolev spaces that
	\begin{align*}
		\|\pi_{0*}Bu\|_{H^s(M)} \leq C_1\|Bu\|_{H^s(SM)} \leq C_2 \|u\|_{H^s_+}
	\end{align*}
	where we used $\pi_{0*}: H^s(SM) \to H^s(M)$ is bounded. Therefore we conclude $\pi_{0*}: H^s_+ \to H^s(M)$ is bounded, implying $\pi_{0*}R_{+}(z)\pi_0^*: H^s(M) \to H^s(M)$ is bounded and meromorphic in $z$ and so the proof is complete after repeating the analogous argument for $\pi_{0*}R_{-}(z)\pi_0^*$.
\end{proof}
In particular, we obtain the following stronger results after distinguishing between $z = 0$ and $\re(z) > 0$.
\begin{cor}
	For an Anosov manifold $M$ and $s > 0$ we have
	\begin{enumerate}[(a)]
		\item If $f \in H^s(M)$ satisfies $\int_M f \dvol = 0$, then $\Pi_0^{(z)}f: \C \to H^s(M)$ is holomorphic in a neighborhood of $z = 0$. \label{it:z=0}
		\item $\Pi_0^{(z)}: H^s(M) \to H^s(M)$ is bounded and holomorphic over $\re(z) > 0$. \label{it:z>0}
	\end{enumerate}
	\label{cor:hol}
\end{cor}

\begin{proof}
	Note that if both resolvents $R_{\pm}^{(z)}: H^s_{\pm} \to H^s_{\pm}$ are holomorphic in a neighborhood of some $z_0 \in \C$, then after decomposing $\Pi_0^{(z)}$ by (\ref{eq:resolvent-decomp}), we will find that $\Pi_0^{(z)}$ is holomorphic in this same neighborhood as a map into $H^s(M)$. Thus for \ref{it:z=0}, recall the condition $\int_M f \dvol = 0$ implies $R_{\pm}^{(z)}: H^s_{\pm} \to H^s_{\pm}$ is holomorphic in a neighborhood of $z = 0$ and therefore $\Pi_0^{(z)}f: \C \to H^s(M)$ is holomorphic in this same neighborhood. For \ref{it:z>0}, simply recall $R_{\pm}^{(z)}: H^s_{\pm} \to H^s_{\pm}$ is holomorphic for $\re(z) > 0$.
\end{proof}

We are nearly ready to prove Theorem~\ref{thm:inv-formula}. However, the nature of the inversion formula of Theorem~\ref{thm:M-att-inv} requires that we show the \emph{joint} continuity of the operators $S^{(z)}$ and $L^{(z)}$ in order to properly pass to the limit $z \to 0$. For this, we recall the following fact from functional analysis, which 
%follows quickly from the uniform boundedness principle and 
allows us to conclude joint continuity from separate continuity.

%TODO? cite and simplify below (or for proof see "writeup" note)
\begin{lemma}
	Let $X$ and $Y$ be Banach spaces, let $S$ be a metric space, and consider a family $A: S \times X \to Y$ of bounded linear operators $X \to Y$ parametrized by $S$. If $A(\boldsymbol{\cdot}, x_0): S \to Y$ is continuous at $s_0$, then $A$ is jointly continuous at $(s_0, x_0)$.
	\label{lem:jointly-cts}
\end{lemma}

\begin{proof}
	Let $x_n \to x_0$ and $s_n \to s_0$, then note $\sup_{n}\|A(s_n, x)\| < \infty$ for each $x \in X$ by continuity in $s$, so by uniform boundedness $\sup_{n \in \N}\|A(s_n, \cdot)\| < \infty$. Therefore:
	\begin{align*}
		\|A(s_n, x_n) - A(s,x)\|
		&\leq \|A(s_n, x_n) - A(s_n, x)\| + \|A(s_n, x) - A(s,x)\|\\
		&\leq \sup_{n \in \N}\|A(s_n, \cdot)\| \cdot \|x_n - x\| + \|A(s_n, x) - A(s, x)\| \to^{n \to \infty} 0. 
	\end{align*}
\end{proof}

\begin{proof}[Proof of Theorem~\ref{thm:inv-formula}]
	By Lemma~\ref{lem:norm} we may assume without loss of generality that $M$ has constant curvature $-1$, and by Theorem~\ref{thm:M-att-inv} we have in this case the inversion formula
	\begin{equation}
		(\Delta - z(z+1)) S^{(z)} \Pi^{(z)}_0 f = -8\pi^2 f
		\label{eq:take-limit-of}
	\end{equation}
	for all $\re(z) > 0$. We will show that as $z \to 0$ we have $(\Delta - z(z+1)) S^{(z)} \Pi^{(z)}_0 f \to \Delta S \Pi_0 f$ in $H^s(M)$ for any $s > 0$ and thus we recover the Theorem~\ref{thm:inv-formula} inversion formula after taking the limit of both sides of (\ref{eq:take-limit-of}). For this we need continuity at $z = 0$ and in particular we will show there is a neighborhood $U \subset \C$ of $0$ such that
	\begin{enumerate}[(a)]
		\item $\Pi_0^{(z)}f: U \to H^s(M)$ is continuous,\label{it:Pi-cont}
		\item $S^{(z)}: U \times H^s(M) \to H^s(M)$ is jointly continuous,\label{it:S-cont}
		\item $L^{(z)}: U \times H^s(M) \to H^{s-2}(M)$ is jointly continuous.\label{it:L-cont}
	\end{enumerate}
	Notice \ref{it:Pi-cont} immediately follows from Corollary~\ref{cor:hol}. In order to show joint continuity, we use Lemma~\ref{lem:jointly-cts}. Indeed, using $S^{(z)} = \f{1}{2}\Pi_0^{(z+1)}$, Corollary~\ref{cor:hol} promises a family of bounded operators $S^{(z)}: H^{s}(M) \to H^{s}(M)$ that is holomorphic in $z$ and therefore continuous, so we have \ref{it:S-cont}. Finally, because $L^{(z)} = \Delta - z(z+1)$, the continuity of $L^{(z)}$ follows from the continuity of multiplication by $z(z+1)$, giving \ref{it:L-cont}.
\end{proof}

%TODO ....

%TODO explain 1 \times 1 earlier...

Finally, we describe how the Theorem~\ref{thm:inv-formula} inversion formula provides an explicit construction for invariant distributions with prescribed pushforward over closed hyperbolic surfaces.
\begin{cor}
	Let $M$ be a closed surface of constant curvature $K < 0$ and take $f \in C^{\infty}(M)$. Then an explicit solution to $Xw = 0$ with $\pi_{0*}w = f$ is given by
	\begin{equation}
		w = - (8\pi^2)^{-1} \Pi \pi_{0}^* \Delta S_K (f - \ol{f}) + \f{\ol{f}}{2\pi}
		\label{eq:inv-dist}
	\end{equation}
	where $\ol{f}$ denotes the average value of $f$. Furthermore, $w \in H^r(M)$ for all $r < 0$.
	\label{cor:inv-dist}
\end{cor}
%TODO mention psi-do somewhere, and self-adjoint and Pi 1 = 0 and XPi = 0
\begin{proof}
	First note by $X\Pi = 0$ and $\ol{f}/2\pi$ a constant, we immediately get the invariance $Xw = 0$. For the pushforward property, first consider the operator $\Pi_0 + \pi_{0*} 1 \otimes 1 \pi_{0*}: H^s(M) \to H^{s+1}(M)$ which is injective by \cite[Lemma 4.6]{2021-gouzel-lefeuvre} and, as noted in the proof of \cite[Theorem 4.7]{2021-gouzel-lefeuvre}, this operator is Fredholm of index $0$, so therefore is bijective. Using $\Pi_0$ self-adjoint and $\Pi 1 = 0$, we find $\Pi_0$ maps to functions with average value zero: $\ip{\Pi_0 h}{1}_{L^2(M)} = \ip{h}{\Pi_0 1}_{L^2(M)} = 0$. This together with $\Pi_0$ an elliptic pseudodifferential operator implies $\Pi_0: C_{\diamond}^{\infty}(M) \to C_{\diamond}^{\infty}(M)$ is a bijection where
%the space %optional
$C_{\diamond}^{\infty}(M) = \{h \in C^{\infty}(M) : \ip{h}{1}_{L^2(M)} = 0\}$.
%denotes smooth functions with average value $0$.
Thus the left inverse $(8\pi^2)^{-1}\Delta S_K$ to $\Pi_0$ provided in the Theorem~\ref{thm:inv-formula} inversion formula is also a right inverse, so after applying $\pi_{0*}$  to (\ref{eq:inv-dist}) we get $\pi_{0*}w = f$. The regularity of $w$ follows directly from the mapping properties of $\Pi$.
\end{proof}

\section*{Acknowledgements}
Many thanks to Gabriel Paternain for countless helpful conversations. This material is based upon work supported by the National Science Foundation Graduate Research Fellowship under Grant No. DGE-2140004.

%BIBLIOGRAPHY
%\bibliographystyle{plain}
\bibliographystyle{alpha}
\bibliography{bib}

\end{document}